\documentclass[12pt]{article}      
\usepackage{amsmath,amssymb,amsthm, amscd, verbatim, setspace}
\usepackage{graphicx, subfigure}
\usepackage{subfig}
\usepackage{cite}

\theoremstyle{plain}
\newtheorem{thm}{Theorem}
\newtheorem{lem}[thm]{Lemma}
\newtheorem{cor}[thm]{Corollary}

\title{Convex geometric $(k+2)$-quasiplanar representations of semi-bar $k$-visibility graphs}
\author{Jesse Geneson, Tanya Khovanova, Jonathan Tidor}
\date{}
\begin{document}
\maketitle

\begin{abstract}
We examine semi-bar visibility graphs in the plane and on a cylinder in which sightlines can pass through $k$ objects. We show every semi-bar $k$-visibility graph has a $(k+2)$-quasiplanar representation in the plane with vertices drawn as points in convex position and edges drawn as segments. We also show that the graphs having cylindrical semi-bar $k$-visibility representations with semi-bars of different lengths are the same as the $(2k+2)$-degenerate graphs having edge-maximal $(k+2)$-quasiplanar representations in the plane with vertices drawn as points in convex position and edges drawn as segments. 
\end{abstract}

\section{Introduction} 
Bar visibility graphs are graphs for which vertices can be drawn as horizontal segments (bars) and edges can be drawn as vertical segments (sightlines) so that two bars are visible to each other if and only if there is a sightline which intersects them and no other bars. The study of bar visibility graphs was motivated in part by the problem of efficiently designing very large scale integration (VLSI) circuits \cite{Luccio}. Past research has shown how to represent planar graphs and plane triangular graphs as bar visibility graphs \cite{Duchet, bartri}. Dean \emph{et al.} \cite{AlDeank} introduced a generalization of bar visibility graphs in which bars are able to see through at most $k$ other bars for some nonnegative integer $k$. These graphs are known as bar $k$-visibility graphs. 

We study bar $k$-visibility graphs in which every bar has left endpoint on the $y$-axis. Such bars are called \emph{semi-bars}. We also consider semi-bar $k$-visibility graphs obtained by placing the semi-bars on the surface of a cylinder with each semi-bar parallel to the cylinder's axis of symmetry. Felsner and Massow \cite{Felsner} proved that the maximum number of edges in any semi-bar $k$-visibility graph with $n$ vertices is $(k+1)(2n-2k-3)$ for $n \geq 2k+2$ and $\binom{n}{2}$ for $n \leq 2k+2$. Similar bounds were derived by Capoyleas and Pach \cite{CP} when they proved that the maximum number of straight line segments in the plane connecting $n$ points in convex position so that no $k+2$ segments are pairwise crossing is $\binom{n}{2}$ for $n \leq 2k+2$ and $2(k+1) n - \binom{2k+3}{2}$ for $n \geq 2k+2$. 

We prove that every semi-bar or cylindrical semi-bar $k$-visibility graph can be represented in the plane with vertices drawn as points in convex position and edges drawn as segments so there are no $k+2$ pairwise crossing edges. Furthermore, we prove that the class of graphs having cylindrical semi-bar $k$-visibility representations with semi-bars of different lengths is the same as the class of $(2k+2)$-degenerate graphs having edge-maximal $(k+2)$-quasiplanar representations in the plane with vertices drawn as points in convex position and edges drawn as segments. Section~\ref{sec:order} contains a more detailed description of the results.

\section{Definitions}

A \emph{semi-bar $k$-visibility representation} of a graph $G = (V, E)$ is a collection $\left\{s_{v}\right\}_{v \in V}$ of disjoint segments in the plane parallel to the $x$-axis with left endpoints on the $y$-axis such that for all $a, b \in V$ there is an edge $\left\{a,b\right\} \in E$ if and only if there exists a vertical segment (a \emph{sightline}) which intersects $s_{a}$, $s_{b}$, and at most $k$ other semi-bars. A graph is a \emph{semi-bar $k$-visibility graph} if it has a semi-bar $k$-visibility representation. 

A \emph{cylindrical semi-bar $k$-visibility representation} of a graph $G = (V, E)$ is a collection $\left\{s_{v}\right\}_{v \in V}$ of disjoint segments parallel to the $x$-axis in three dimensions with left endpoints on the circle $\left\{(0, y, z): y^{2}+z^{2} = 1 \right\}$ such that for all $a, b \in V$ there is an edge $\left\{a,b\right\} \in E$ if and only if there exists a circular arc along the surface of the cylinder parallel to the $y z$ plane (a \emph{sightline}) which intersects $s_{a}$, $s_{b}$, and at most $k$ other semi-bars. A graph is a \emph{cylindrical semi-bar $k$-visibility graph} if it has a cylindrical semi-bar $k$-visibility representation. Figure~\ref{cylbar} shows a cylindrical semi-bar visibility graph and a two-dimensional view of a corresponding representation in which bars are represented by radial segments and sightlines are represented by arcs.

\begin{figure}[t]
\centering
\mbox{\subfigure{\includegraphics[scale = 0.3]{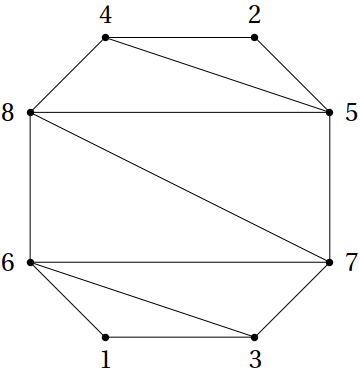}}\quad
\subfigure{\includegraphics[scale = 0.375]{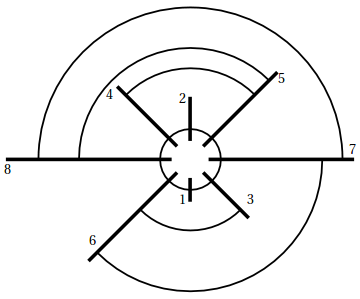} }}
\caption{A cylindrical semi-bar visibility graph and a corresponding representation.} \label{cylbar}
\end{figure}

A graph is \emph{$k$-quasiplanar} if it can be drawn in the plane with no $k$ pairwise crossing edges. For example $2$-quasiplanar graphs are planar. Call a $k$-quasiplanar graph $G$ \emph{convex geometric} if it has a $k$-quasiplanar representation $C_{G}$ with vertices drawn as points in convex position and edges drawn as segments. Call a $k$-quasiplanar convex geometric representation \emph{maximal} if adding any straight edge to the representation causes it to have $k$ pairwise crossing edges. 

In a set of points in the plane, call a pair of points a \emph{$j$-pair} if the line through those points has exactly $j$ points on one side. Every maximal $(k+2)$-quasiplanar convex geometric representation has edges between all $j$-pairs in the representation for each $j \leq k$.

A graph is called \emph{$l$-degenerate} if all of its subgraphs contain a vertex of degree at most $l$. Cylindrical semi-bar $k$-visibility graphs are $(2k+2)$-degenerate for all $k \geq 0$ since the shortest semi-bar in any subset of semi-bars sees at most $2k+2$ other semi-bars, so cylindrical semi-bar $k$-visibility graphs have chromatic number at most $2k+3$ and clique number at most $2k+3$. Furthermore, Felsner and Massow \cite{Felsner} showed $K_{2k+3}$ is a semi-bar $k$-visibility graph, so $K_{2k+3}$ is also a cylindrical semi-bar $k$-visibility graph. Thus $2k+3$ is the maximum possible chromatic number and clique number of cylindrical semi-bar $k$-visibility graphs.

\section{Order of results}\label{sec:order}

In Section~\ref{sec:sbkv} we show every cylindrical semi-bar $k$-visibility graph is a $(k+2)$-quasiplanar convex geometric graph. In particular, every cylindrical semi-bar $k$-visibility graph with a representation having semi-bars of different lengths has a maximal $(k+2)$-quasiplanar convex geometric representation. Furthermore, we show that if a semi-bar $k$-visibility representation $R$ with semi-bars of different lengths is curled into a cylindrical semi-bar $k$-visibility representation $R'$, then the graphs corresponding to $R$ and $R'$ will be the same if and only if the top $k+1$ and bottom $k+1$ semi-bars in $R$ comprise the longest $2k+2$ semi-bars in $R$ and the longest $2k+2$ semi-bars in $R$ are all visible to each other. 

In Section~\ref{sec:kqcg} we show every graph with a maximal planar convex geometric representation can be represented as a semi-bar or cylindrical semi-bar visibility graph with semi-bars of different lengths. Moreover, we show every $(2k+2)$-degenerate graph with a maximal $(k+2)$-quasiplanar convex geometric representation has a cylindrical semi-bar $k$-visibility representation with semi-bars of different lengths. For each $k \geq 1$ we also exhibit $(k+2)$-quasiplanar convex geometric graphs which are not subgraphs of any cylindrical semi-bar $k$-visibility graph. Furthermore, we exhibit maximal $(k+2)$-quasiplanar convex geometric representations which cannot be transformed, without changing cyclic positions of vertices, into cylindrical semi-bar $k$-visibility representations having semi-bars of different lengths with a pair of adjacent semi-bars among the $2k+3$ longest semi-bars.

\section{Cylindrical semi-bar $k$-visibility graphs}\label{sec:sbkv}
We show that every cylindrical semi-bar $k$-visibility graph with $n$ semi-bars and $m$ edges can be represented by $n$ points in the plane with $m$ segments between them such that there are no $k+2$ pairwise crossing segments. This gives an alternative proof for the upper bound on the maximum number of edges in a semi-bar $k$-visibility graph with $n$ vertices in conjunction with the upper bound in \cite{CP}. 

\begin{thm}\label{sbtokc}
Every cylindrical semi-bar $k$-visibility graph is a $(k+2)$-qua\-si\-pla\-nar convex geometric graph.
\end{thm}

\begin{proof}
Let $H$ be a cylindrical semi-bar $k$-visibility graph with $n$ vertices. Let $B_{H}$ be a cylindrical semi-bar $k$-visibility representation of $H$ and let $b_{1}, b_{2}, \ldots, b_{n}$ be the semi-bars in $B_{H}$ listed in cyclic order. Let $p_{1}, \ldots, p_{n}$ be any points in the plane in convex position listed in cyclic order. Draw a segment between $p_{i}$ and $p_{j}$ if and only if there is an edge in $H$ between the vertices corresponding to $b_{i}$ and $b_{j}$. For contradiction suppose that the resulting drawing contains a set $C$ of $k+2$ pairwise crossing edges. None of the edges in $C$ have any vertices in common, or else they would not be crossing.

Consider an edge $e \in C$ between two points $p_{i}$ and $p_{j}$ for which $i < j$ and $b_{i}$ or $b_{j}$ is a semi-bar of minimal length among all semi-bars corresponding to points in edges in $C$. Since the other $k+1$ edges in $C$ cross $e$, then each edge in $C$ besides $e$ has some endpoint $p_{t_{1}}$ for which $i < t_{1} < j$ and another endpoint $p_{t_{2}}$ for which $t_{2} < i$ or $j < t_{2}$. Since $H$ has an edge between the vertices corresponding to $b_{i}$ and $b_{j}$, then $b_{i}$ and $b_{j}$ are both longer than all but at most $k$ of the semi-bars in at least one of the half-spaces bounded by the plane through $b_{i}$ and $b_{j}$. Therefore, there exists $t'$ for which $b_{t'}$ is shorter than both $b_{i}$ and $b_{j}$, and $p_{t'}$ is an endpoint of one of the edges in $C$. This is a contradiction.
\end{proof}

The last theorem showed every subgraph of a cylindrical semi-bar $k$-visibility graph is a $(k+2)$-quasiplanar convex geometric graph. In particular, every semi-bar $k$-visibility graph is a $(k+2)$-quasiplanar convex geometric graph.

\begin{cor}
The maximum number of edges in a cylindrical semi-bar $k$-visibility graph with $n$ vertices is $(k+1)(2n-2k-3)$ for $n \geq 2k+2$ and $\binom{n}{2}$ for $n \leq 2k+2$.
\end{cor}

If the semi-bars in a cylindrical semi-bar $k$-visibility representation have different lengths, then the representation will have the maximum possible number of edges. The proof of this fact is similar to the proof of the upper bound for the number of edges in semi-bar $k$-visibility graphs in \cite{Felsner}.

\begin{lem}
The exact number of edges in every graph on $n$ vertices having a cylindrical semi-bar $k$-visibility representation with semi-bars of different lengths is $(k+1)(2n-2k-3)$ for $n \geq 2k+2$ and $\binom{n}{2}$ for $n \leq 2k+2$.
\end{lem}

\begin{proof}
If $n \leq 2k+2$, then the graph is complete. For $n > 2k+2$ we count, for each semi-bar $b$, how many edges in the cylindrical semi-bar $k$-visibility representation have $b$ as the shorter semi-bar. If $b$ is not among the $2k+2$ longest semi-bars, then there are $2k+2$ such edges. If $b$ is the $i^{th}$ longest semi-bar for some $1 \leq i \leq 2k+2$, then there are $i-1$ edges having $b$ as the shorter semi-bar.
\end{proof}

If the semi-bars in the representation have different lengths, then we also show that the construction in Theorem~\ref{sbtokc} yields a maximal $(k+2)$-quasiplanar convex geometric representation.

\begin{thm}
Every cylindrical semi-bar $k$-visibility graph with a representation having semi-bars of different lengths has a maximal $(k+2)$-quasiplanar convex geometric representation.
\end{thm}

\begin{proof}
Every cylindrical semi-bar $k$-visibility graph $G$ on $n$ vertices with a representation having semi-bars of different lengths has $(k+1)(2n-2k-3)$ edges for $n \geq 2k+2$ and $\binom{n}{2}$ edges for $n \leq 2k+2$. Furthermore, every $(k+2)$-quasiplanar convex geometric graph has at most $(k+1)(2n-2k-3)$ edges for $n \geq 2k+2$ and $\binom{n}{2}$ edges for $n \leq 2k+2$. In Theorem~\ref{sbtokc} we proved $G$ has a $(k+2)$-quasiplanar convex geometric representation. Thus $G$ has a maximal $(k+2)$-quasiplanar convex geometric representation.
\end{proof}

Any semi-bar $k$-visibility representation $R$ can be curled into a cylindrical semi-bar $k$-visibility representation by placing the semi-bars in $R$ on the cylinder $y^{2}+z^{2} = 1$ in the same order in which they occur in $R$. This transformation never deletes sightlines, but it may add sightlines. If a semi-bar visibility representation $R$ with semi-bars of different lengths is curled into a cylindrical semi-bar visibility representation $R'$, then the graphs corresponding to $R$ and $R'$ will be the same if and only if the top and bottom semi-bars in $R$ are the two longest semi-bars in $R$. The next lemma generalizes the previous statement to $k$-visibility for any $k > 0$.

\begin{lem}
If a semi-bar $k$-visibility representation $R$ with semi-bars of different lengths is curled into a cylindrical semi-bar $k$-visibility representation $R'$, then the graphs corresponding to $R$ and $R'$ will be the same if and only if the top $k+1$ and bottom $k+1$ semi-bars in $R$ comprise the longest $2k+2$ semi-bars in $R$ and the longest $2k+2$ semi-bars in $R$ are all visible to each other. 
\end{lem}

\begin{proof}
If $R$ is a semi-bar $k$-visibility representation with semi-bars of different lengths such that the top $k+1$ and bottom $k+1$ semi-bars in $R$ comprise the longest $2k+2$ semi-bars in $R$ and the longest $2k+2$ semi-bars in $R$ are visible to each other, then there is no sightline in $R'$ which is not in $R$. Indeed, if there was a sightline between two semi-bars $b_{1}$ and $b_{2}$ in $R'$ which was not in $R$, then $b_{1}$ and $b_{2}$ would not both be among the longest $2k+2$ semi-bars. Then the sightline between $b_{1}$ and $b_{2}$ in $R'$ would pass through $k+1$ semi-bars that were the top $k+1$ semi-bars from $R$ or the bottom $k+1$ semi-bars from $R$. This contradicts the definition of $k$-visibility, so the graphs corresponding to $R$ and $R'$ have the same edges. 

Conversely, if the graphs corresponding to $R$ and $R'$ have the same edges, then the longest $2k+2$ semi-bars in $R$ can all see each other. Suppose for contradiction that the top $k+1$ and bottom $k+1$ semi-bars in $R$ do not comprise the longest $2k+2$ semi-bars in $R$. Without loss of generality assume that there is a semi-bar $t$ among the top $k+1$ semi-bars in $R$ that is not one of the $2k+2$ longest semi-bars in $R$. Let $b$ be the bottom semi-bar in $R$. Since there is a sightline between $b$ and $t$ in $R'$, then there must be a sightline between $b$ and $t$ in $R$. However the sightline between $b$ and $t$ in $R$ crosses at least $k+1$ semi-bars longer than $t$, a contradiction. Thus the top $k+1$ and bottom $k+1$ semi-bars in $R$ comprise the longest $2k+2$ semi-bars in $R$.
\end{proof}

\section{Maximal $(k+2)$-quasiplanar convex geometric representations}\label{sec:kqcg}
In this section we find cylindrical semi-bar $k$-visibility representations of $(2k+2)$-degenerate graphs with maximal $(k+2)$-quasiplanar convex geometric representations. First we show that every graph with a maximal planar convex geometric representation is a semi-bar visibility graph.

\begin{figure}[t]
\centering
\mbox{\subfigure{\includegraphics[scale = 0.3]{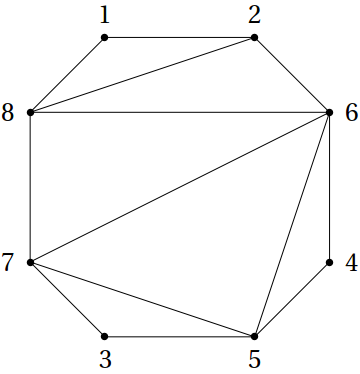}}\quad
\subfigure{\includegraphics[scale = 0.68]{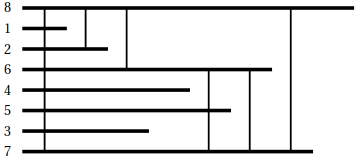} }}
\caption{A maximal planar convex geometric graph and a corresponding semi-bar visibility representation.} \label{sbar}
\end{figure}

\begin{lem} \label{mptosb}
Every graph with a maximal planar convex geometric representation has a semi-bar visibility representation with semi-bars of different lengths.
\end{lem}

\begin{proof}
A cylindrical semi-bar visibility representation with the two longest semi-bars adjacent to each other can be changed into a semi-bar visibility representation by cutting the cylinder between the two longest semi-bars parallel to the cylinder's axis of symmetry and then flattening the representation. We show every graph with a maximal planar convex geometric representation has a cylindrical semi-bar visibility representation with the two longest semi-bars adjacent to each other.

Let $G$ be a graph with $n$ vertices having a maximal planar convex geometric representation $C_{G}$. Pick any vertex in $C_{G}$ to be $v_{1}$ and name the vertices $v_{1}, v_{2}, \ldots, v_{n}$ in cyclic order around $C_{G}$. We inductively construct a cylindrical semi-bar visibility representation $B_{G}$ from $C_{G}$ with $n$ semi-bars $b_{1}, b_{2}, \ldots, b_{n}$ listed in cyclic order around the cylinder. For each $i$ and $j$ there will be a sightline between $b_{i}$ and $b_{j}$ in $B_{G}$ if and only if there is a segment between $v_{i}$ and $v_{j}$ in $C_{G}$. Furthermore, no two semi-bars will have the same length and each length will be one of the integers $1, 2, \ldots, n$. 

The semi-bar $b_{1}$ is assigned the length $n$ and the semi-bar $b_{n}$ is assigned $n-1$. The remaining semi-bars will be assigned lengths $1, 2, \ldots, n-2$ from shortest to longest. We may assume that $n > 2$ or else the proof is already finished. Let $C_{0} = C_{G}$. As $C_{0}$ is a maximal planar convex geometric representation, then every interior face in $C_{0}$ is a triangle. If $C_{0}$ has at least four vertices, then $C_{0}$ has at least two triangular faces, so there is a vertex in $C_{0}$ besides $v_{1}$ or $v_{n}$ with only two neighbors in $C_{0}$. If $C_{0}$ has three vertices, then every vertex in $C_{0}$ has exactly two neighbors in $C_{0}$. 

Let $v_{m_{0}}$ be any vertex in $C_{0}$ besides $v_{1}$ or $v_{n}$ with only two neighbors in $C_{0}$. The semi-bar $b_{m_{0}}$ is assigned the length $1$. Then $b_{m_{0}}$ will see the semi-bars corresponding to the neighbors of $v_{m_{0}}$ in $C_{0}$ no matter which lengths greater than $1$ are assigned to those semi-bars. Let $C_{1}$ be obtained by deleting $v_{m_{0}}$ and any edges including $v_{m_{0}}$ from $C_{0}$. Then $C_{1}$ is a maximal planar convex geometric representation.

Continuing by induction after $i$ iterations, suppose that $C_{i}$ is a maximal planar convex geometric representation. Then every interior face in $C_{i}$ is a triangle. At least one of the triangles has a vertex besides $v_{1}$ or $v_{n}$ with only two neighbors in $C_{i}$, so let $v_{m_{i}}$ be such a vertex. The semi-bar $b_{m_{i}}$ is assigned the length $i+1$. Then $b_{m_{i}}$ will see the semi-bars corresponding to the neighbors of $v_{m_{i}}$ in $C_{i}$ no matter which lengths greater than $i+1$ are assigned to those semi-bars. 

Let $C_{i+1}$ be obtained by deleting $v_{m_{i}}$ and any edges including $v_{m_{i}}$ from $C_{i}$. Then $C_{i+1}$ is a maximal planar convex geometric representation. Thus after $n-2$ iterations we obtain a cylindrical semi-bar visibility representation of $G$ with the two longest semi-bars next to each other. Figure~\ref{sbar} shows a maximal planar convex geometric graph with $8$ vertices and a corresponding semi-bar visibility representation in which the top and bottom semi-bars have lengths $8$ and $7$. The lengths of the other semi-bars are assigned according to the description in this proof. 
\end{proof}

\begin{cor}
For all graphs $G$, $G$ has a semi-bar visibility representation with semi-bars of different lengths and an edge between the topmost and bottommost semi-bars if and only if $G$ has a maximal planar convex geometric representation.
\end{cor}

Lemma~\ref{mptosb} implies every planar convex geometric graph is a subgraph of a cylindrical semi-bar visibility graph. For $k \geq 1$ we exhibit $(k+2)$-quasiplanar convex geometric graphs which are not subgraphs of any cylindrical semi-bar $k$-visibility graph. 

\begin{lem}
For all $k \geq 1$ there exist $(k+2)$-quasiplanar convex geometric graphs which are not subgraphs of any cylindrical semi-bar $k$-visibility graph.
\end{lem}

\begin{proof}
Fix $k \geq 1$. Let $D$ be a drawing with $4(k+1)$ vertices $v_{1}, v_{2}, \ldots, v_{4(k+1)}$ in convex position listed in cyclic order with straight edges $\left\{v_{i}, v_{3(k+1)+1-i} \right\}$ and $\left\{v_{k+1+i}, v_{4(k+1)+1-i} \right\}$ for each $1 \leq i \leq k+1$. Let $D'$ be any maximal $(k+2)$-quasiplanar convex geometric representation with $4(k+1)$ vertices which contains $D$. For any vertices $u$ and $v$ for which $\left\{u,v\right\}$ is a $j$-pair in $D'$ for some $j \leq k$, there is an edge between $u$ and $v$ in $D'$ since $D'$ is maximal. Then each vertex in $D'$ has $2(k+1)$ edges which are $j$-pairs for some $j \leq k$ and another edge from $D$ which is a $j$-pair for some $j > k$. Thus every vertex in $D'$ has degree at least $2(k+1)+1$, but cylindrical semi-bar $k$-visibility graphs are $(2k+2)$-degenerate. So the graph represented by $D'$ is not a subgraph of any cylindrical semi-bar $k$-visibility graph. See Figure~\ref{keq1} for an example in the case $k = 1$.
\end{proof}

\begin{figure}[t]
\begin{center}
\includegraphics[scale=0.3]{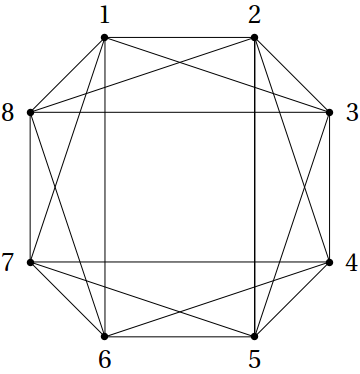}
\setlength{\abovecaptionskip}{0pt}
\caption{A $3$-quasiplanar convex geometric graph which is not $4$-degenerate.}
\label{keq1}
\end{center}
\end{figure}

If a $(2k+2)$-degenerate graph $G$ has a maximal $(k+2)$-quasiplanar convex geometric representation, then we can show that $G$ is a cylindrical semi-bar $k$-visibility graph using a proof like the one for Lemma~\ref{mptosb}.

\begin{thm}\label{final}
Every $(2k+2)$-degenerate graph with a maximal $(k+2)$-quasiplanar convex geometric representation has a cylindrical semi-bar $k$-visibility representation with semi-bars of different lengths.
\end{thm}

\begin{proof}
Let $G$ be a $(2k+2)$-degenerate graph with $n$ vertices having a maximal $(k+2)$-quasiplanar convex geometric representation $C_{G}$. Pick any vertex in $C_{G}$ to be $v_{1}$ and name the vertices $v_{1}, v_{2}, \ldots, v_{n}$ in cyclic order around $C_{G}$. We construct a cylindrical semi-bar $k$-visibility representation $B_{G}$ from $C_{G}$ with $n$ semi-bars $b_{1}, b_{2}, \ldots, b_{n}$ listed in cyclic order around the cylinder. For each $i$ and $j$ there will be a sightline between $b_{i}$ and $b_{j}$ in $B_{G}$ if and only if there is a segment between $v_{i}$ and $v_{j}$ in $C_{G}$. 

Furthermore, no two semi-bars will have the same length and each length will be one of the integers $1, 2, \ldots, n$. The semi-bars will be assigned lengths from shortest to longest. Let $C_{0} = C_{G}$. Each $C_{i}$ will be a maximal $(k+2)$-quasiplanar convex geometric representation.

Since $G$ is a $(2k+2)$-degenerate graph, then $C_{0}$ has a vertex with at most $2k+2$ neighbors in $C_{0}$, so let $v_{m_{0}}$ be such a vertex. Since $C_{0}$ is a maximal $(k+2)$-quasiplanar convex geometric representation, then the neighbors of $v_{m_{0}}$ in $C_{0}$ are the vertices $v$ for which $\left\{v, v_{m_{0}} \right\}$ is a $j$-pair in $C_{0}$ for each $j \leq k$. The semi-bar $b_{m_{0}}$ is assigned the length $1$. Then $b_{m_{0}}$ will see the semi-bars corresponding to the neighbors of $v_{m_{0}}$ in $C_{0}$ no matter which lengths greater than $1$ are assigned to those semi-bars. 

Let $C_{1}$ be obtained by deleting $v_{m_{0}}$ and any edges including $v_{m_{0}}$ from $C_{0}$. If we suppose for contradiction that $C_{1}$ is not a maximal $(k+2)$-quasiplanar convex geometric representation, then there is an edge which can be added to $C_{1}$ so that the resulting representation is still a $(k+2)$-quasiplanar convex geometric representation. However, this same edge can be added to $C_{0}$ so that the resulting representation is still a $(k+2)$-quasiplanar convex geometric representation since no edge containing $v_{m_{0}}$ can cross more than $k$ other edges in $C_{0}$. This contradicts the hypothesis that $C_{0}$ is a maximal $(k+2)$-quasiplanar convex geometric representation. Thus $C_{1}$ is a maximal $(k+2)$-quasiplanar convex geometric representation.

Continuing by induction after $i$ iterations suppose that $C_{i}$ is a maximal $(k+2)$-quasiplanar convex geometric representation of a $(2k+2)$-degenerate graph. Then $C_{i}$ has a vertex with at most $2k+2$ neighbors in $C_{i}$, so let $v_{m_{i}}$ be such a vertex. Since $C_{i}$ is a maximal $(k+2)$-quasiplanar convex geometric representation, then the neighbors of $v_{m_{i}}$ in $C_{i}$ are the vertices $v$ for which $\left\{v, v_{m_{i}} \right\}$ is a $j$-pair in $C_{i}$ for each $j \leq k$. The semi-bar $b_{m_{i}}$ is assigned the length $i+1$. Then $b_{m_{i}}$ will see the semi-bars corresponding to the neighbors of $v_{m_{i}}$ in $C_{i}$ no matter which lengths greater than $i+1$ are assigned to those semi-bars. 

Let $C_{i+1}$ be obtained by deleting $v_{m_{i}}$ and any edges including $v_{m_{i}}$ from $C_{i}$. Then $C_{i+1}$ is a maximal $(k+2)$-quasiplanar convex geometric representation. Thus after $n$ iterations we obtain a cylindrical semi-bar $k$-visibility representation of $G$.
\end{proof}

\begin{cor}
For all graphs $G$, $G$ has a cylindrical semi-bar $k$-visibility representation with semi-bars of different lengths if and only if $G$ is $(2k+2)$-degenerate and has a maximal $(k+2)$-quasiplanar convex geometric representation.
\end{cor}

\begin{figure}[t]
\begin{center}
\includegraphics[scale=0.3]{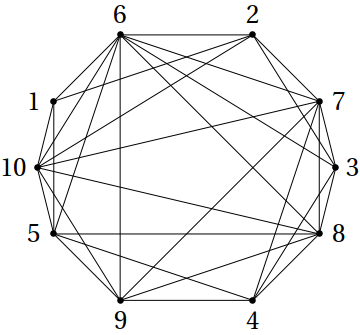}
\setlength{\abovecaptionskip}{0pt}
\caption{A maximal $3$-quasiplanar convex geometric representation corresponding to a cylindrical semi-bar $1$-visibility representation with semi-bars of lengths $1, 6, 2, 7, 3, 8, 4, 9, 5, 10$ in cyclic order.}
\label{m3q}
\end{center}
\end{figure}

\begin{figure}[t]
\centering
\mbox{\subfigure{\includegraphics[scale = 0.25]{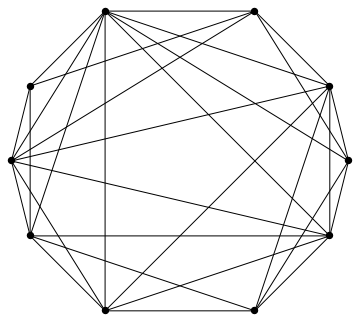}}\quad
\subfigure{\includegraphics[scale = 0.25]{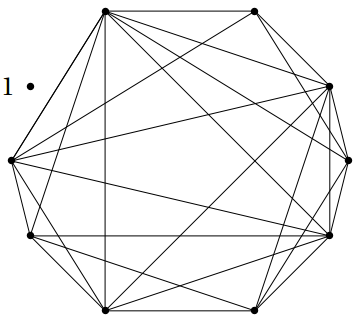} }
\subfigure{\includegraphics[scale = 0.25]{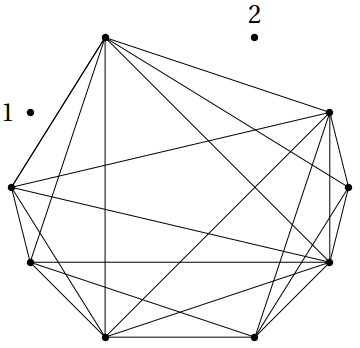} }}
\mbox{\subfigure{\includegraphics[scale = 0.25]{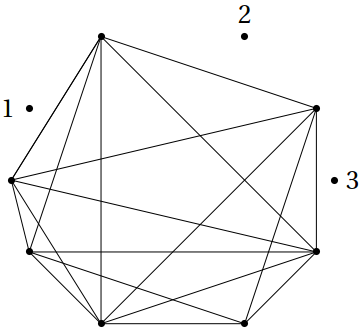} }
\subfigure{\includegraphics[scale = 0.25]{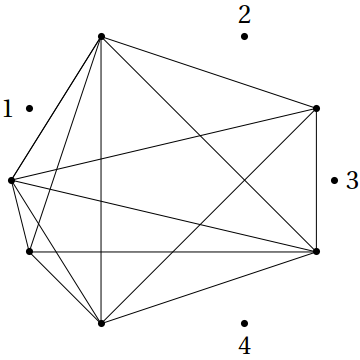} }
\subfigure{\includegraphics[scale = 0.25]{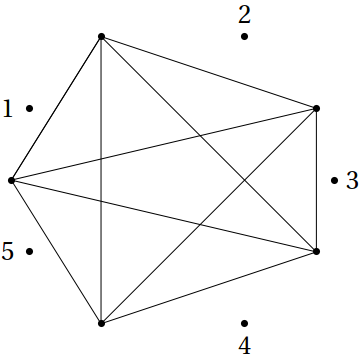} }}
\caption{A maximal $3$-quasiplanar convex geometric representation which cannot be transformed, without changing cyclic positions of vertices, into a cylindrical semi-bar $1$-visibility representation having semi-bars of different lengths with a pair of adjacent semi-bars among the $5$ longest semi-bars.} \label{counter}
\end{figure}

Lemma~\ref{mptosb} showed that any maximal planar convex geometric representation can be transformed without changing cyclic positions of vertices into a cylindrical semi-bar visibility representation having semi-bars of different lengths with the two longest semi-bars adjacent. However for $k \geq 1$ there exist maximal $(k+2)$-quasiplanar convex geometric representations which cannot be transformed, without changing cyclic positions of vertices, into cylindrical semi-bar $k$-visibility representations having semi-bars of different lengths with a pair of adjacent semi-bars among the $2k+3$ longest semi-bars.

For each $k \geq 1$ consider the $(k+2)$-quasiplanar representation $C$ obtained by applying Theorem~\ref{sbtokc} to the cylindrical semi-bar $k$-visibility representation $R$ having $(2k+3)(k+1)$ semi-bars with lengths $1$, $2k+4$, $\ldots$, $1+k(2k+3)$, $2$, $2k+5$, $\ldots$, $2+k(2k+3)$, $3$, $2k+6$, $\ldots$, $3+k(2k+3)$, $\ldots$, $2k+3$, $4k+6$, $\ldots$, $(2k+3)(k+1)$ in cyclic order. Let $B$ be a cylindrical semi-bar $k$-visibility representation with semi-bars of lengths $1, 2, \ldots, (2k+3)(k+1)$ constructed from $C$ using Theorem~\ref{final}. Figure~\ref{m3q} shows the representation $C$ when $k = 1$. 

The only vertex of degree $2k+2$ in $C$ is the vertex $x$ corresponding to the semi-bar of length $1$ in $R$, so $x$ has length $1$ in $B$. When $x$ is removed from $C$, the only vertex of degree $2k+2$ in the resulting representation is the vertex $y$ corresponding to the semi-bar of length $2$ in $R$, so $y$ has length $2$ in $B$. Similarly the vertices in $C$ corresponding to semi-bars of lengths $3$, $4$, $\ldots$, $(2k+3)k$ in $R$ have lengths $3$, $4$, $\ldots$, $(2k+3)k$ respectively in $B$. Then no pair of semi-bars among the $2k+3$ longest semi-bars can be adjacent in $B$, unless the cyclic positions of vertices in $C$ are changed in $B$. Figure~\ref{counter} illustrates the process of assigning lengths to $B$ when $k = 1$. 

\section{Acknowledgments}
Jesse Geneson was supported by an NSF graduate research fellowship. He thanks Katherine Bian for helpful comments on this paper.

\end{document}